 \documentclass[preprint,review,12pt]{elsarticle}

\usepackage{amssymb}
\usepackage[active]{srcltx}
\usepackage{a4wide}
\usepackage{graphicx,palatino,pifont,times}
\usepackage{amsmath,amsthm}
\usepackage{amssymb}
\usepackage{amstext}
\usepackage{cite}
\usepackage{epsfig}
\usepackage{enumerate}
 \usepackage{bm}

\newtheorem{theorem}{Theorem}[section]
\newtheorem{lemma}[theorem]{Lemma}

\theoremstyle{definition}
\newtheorem{definition}[theorem]{Definition}
\newtheorem{example}[theorem]{Example}

\theoremstyle{remark}
\newtheorem{remark}[theorem]{Remark}
\newtheorem{corollary}[theorem]{Corollary}

\catcode`\@=11
\def\theequation{\@arabic{\c@section}.\@arabic{\c@equation}}
\catcode`\@=12

\biboptions{sort}

\journal{}

\begin{document}

\begin{frontmatter}


\title{Bivariate functions  of bounded variation: Fractal dimension and fractional integral}
\author{S. Verma}\author{P. Viswanathan\corref{cor1}}
\cortext[cor1]{Corresponding author\; E-mail address: viswa@maths.iitd.ac.in}
\address{Department of Mathematics\\Indian Institute of Technology Delhi \\New Delhi,  India 110016.}
\begin{abstract}

In contrast to the univariate case,  several definitions are available for the notion of bounded variation for a bivariate function. This article is an attempt to study the Hausdorff dimension and  box dimension of  the graph of a continuous function  defined on a rectangular region in  $\mathbb{R}^2$, which is of bounded variation according to some of these approaches.  We show also that the Riemann-Liouville fractional integral of a function of bounded variation in the sense of Arzel\'a is of bounded variation in the same sense. Further, we deduce the Hausdorff dimension and box dimension of the graph of the fractional integral of a bivariate continuous function of bounded variation.

\end{abstract}
\end{frontmatter}
\section{Introduction}\label{BBVsec1}

This paper
is primarily concerned with the concept of bounded variation of a  bivariate function. The notion of  \emph{bounded variation} was originally introduced by Jordan \citep{Jordan} for a real-valued function on a closed bounded interval in $\mathbb{R}$.
The concept of bounded variation stimulated interest  because of
its  properties such as additivity, decomposability into monotone functions,
continuity, differentiability, measurability and integrability. The functions of bounded
variation, for instance,  plays a major role in the study of rectifiable
curves, Fourier series,
integrals and calculus of variations.

\par The motivation for the current work is multifold. The first is the theory of bivariate function of bounded variation, which enjoys interesting connections
with various branches of pure and applied mathematics.  There is no unique suitable way to extend the
notion of variation to a function of more than one variable.
Various approaches to the notion of bounded variation of a multivariate function target to identify  a class of functions having similar properties
 as that of a univariate function of bounded variation. Of the several approaches to the concept of bounded variation for functions of several variables, popular versions are attributed to  Vitali, Hardy, Arzel\'a, Pierpont, Fr\'echet, Tonelli and Hahn.   The reader may refer \citep{James, Raymond, Raymond2} for a comprehensive collection of these seven variants of bounded variation. In fact, new definitions and approaches continue to be introduced for various applications. For more recent generalizations for the concept of total variation of a function, the interested reader may consult  \citep{APST, BCGT, Cas, Ex, VVC1, VVC2, VVC3} and references quoted therein.

\par Among establishing various properties of a function of bounded variation, calculation of fractal dimension of its graph has gained interest in fractal geometry and related fields. In fractal approximation theory, the Hausdorff dimension and box dimension constitute important
quantifiers that need to agree between the constructed approximants and the object
being approximated. For definitions and basic results on various approaches to the notion of fractal dimension, the reader is referred to the popular textbook by Falconer \citep{Fal}. Using the fact that a univariate function of bounded variation can have at most a countable number of discontinuous points and some basic properties of the Hausdorff dimension, it is easy to prove that the Hausdorff dimension of the graph of a univariate function of bounded variation on $[a, b]$ is  $1$, see, for instance, \citep{Fal}.  Supplementing this, recently, Liang proved an elementary and elegant result that the box dimension of the graph of a univariate continuous function of bounded variation is $1$ (See Theorem 1.3, \citep{Liang}). This result acts as the second motivating influence for our work herein. To be precise,  the aforementioned theorem in reference \citep{Liang} stimulated  to ask if an analogous  result for a bivariate function of bounded variation exists. Section \ref{MS1} seeks to show that this is indeed the case, in fact with a suitable interpretation for the notion of bounded variation. For instance, among others, we prove:
\begin{theorem}
If $f :[a,b] \times [c,d] \rightarrow \mathbb{R}$ is continuous and of bounded variation in the sense of Hahn, then the Hausdorff dimension and box dimension of its graph is $2$.
\end{theorem}
As a prelude to this, we need a bivariate analogue of a well-known proposition (See Proposition 11.1,  \citep{Fal}), which is applied to find the bounds for the box dimension of the graph of a univariate continuous function. Although this is a fundamental and natural extension,  we did not find explicitly anywhere in the literature, for which reason we record it in Section \ref{MS1}. Let us note that while univariate functions of bounded variation are relatively easy to dealt with, the multivariate theory is intricate with roots in geometric measure theory. However, our exposition has a different goal, that is, to apply some elementary techniques to study the dimension of the graph of a bivariate function of bounded variation.

\par
Fractional calculus, which can be broadly interpreted as  the theory of derivatives and integrals of fractional (non-integer) order and their diverse applications,  is an older subject dating back nearly 300 years. The literature relevant to fractional calculus is substantial; for a selection,  the reader can refer to an encyclopedic book \citep{Samko}.
Perhaps due mostly to linguistic reasons, there have been efforts
to relate the two apparently diverse areas - fractional calculus and fractal
geometry. Apart from the linguistic reason, researches to connect fractional calculus with fractals were motivated by the need for physical  and geometric interpretations of the fractional order integration and differentiation \citep{Podlubny,Ruan}. In this regard, in \citep{Liang} it has been deduced that the box dimension of  the graph of the (mixed) Riemann-Liouville fractional integral of a continuous function of bounded variation is $1$. Motivated by this, the last section of the current article establishes the Hausdorff dimension and  box dimension of the graph of the Riemann-Liouville fractional integral of a bivariate continuous function of bounded variation.


\section{Background and Preliminaries} \label{BP}
 This section is to set out the background for the current study.

\subsection{Bounded variation in bivariate function}
We recall some preliminary notions and results on bounded variation of a bivariate function which are needed in the sequel; for details, please refer to \citep{James,Raymond}.\\
 Let $f :[a,b] \times [c,d] \rightarrow \mathbb{R}.$  A set of parallels to the axes:
 $$ x= x_i (i=0,1,2,\dots ,m), \quad a=x_0 <x_1 <\dots <x_m=b;$$
 $$ y= y_j (j=0,1,2,\dots ,n),\quad  c=y_0 < y_1 < \dots < y_n=d$$ will be referred to as a \emph{net}. A net partitions $[a,b] \times [c,d]$ into smaller rectangles called  \emph{cells}. Following \citep{James}, the difference operators $ \triangle ,\triangle_{10},\triangle_{01}$ and $\triangle_{11},$ when applied to $f(x_i,y_j),$ are assigned the following meaning:
 $$ \triangle f(x_i,y_i)=f(x_{i+1},y_{i+1})-f(x_i,y_i),$$
 $$ \triangle_{10} f(x_i,y_j)=f(x_{i+1},y_j)-f(x_i,y_j),$$
 $$ \triangle_{01} f(x_i,y_j)=f(x_i,y_{j+1})-f(x_i,y_j),$$
 $$ \triangle_{11} f(x_i,y_j)=\triangle_{10}(\triangle_{01} f(x_i,y_j))=f(x_{i+1},y_{j+1})-f(x_{i+1},y_j)-f(x_i,y_{j+1})+f(x_i,y_j).$$
   Each of these operators  applied to $f(x,y)$ will have a similar interpretation, wherein the increments of $x$ and $y$ involved are greater than zero but otherwise arbitrary.
 \begin{definition}
Let $f :[a,b] \times [c,d] \rightarrow \mathbb{R}.$ We define  the total variation function $\phi(\overline{x})$  as the total variation of $f(\overline{x},y)$ treated  as a function of $y$ alone in the interval $(c,d)$. Further,  $\phi(\overline{x})=+\infty $ if $f(\overline{x},y)$ is of unbounded variation. Similarly, the total variation function $\mu (\overline{y})$ is the total variation of $f(x, \overline{y})$ considered as a function of $x$ alone in the interval $(a,b)$, or as $ +\infty $ if $f(x, \overline{y})$ is of unbounded variation.
\end{definition}

\begin{definition}
 (Vitali-Lebesgue-Fr\'echet-de la Vall\'ee Poussin) \citep{James}.
A function $f: [a,b] \times [c,d] \rightarrow \mathbb{R}$ is said to be of bounded variation in the Vitali sense if for all nets, the sum $$ \sum_{i=0,j=0}^{m-1,n-1} | \triangle_{11} f(x_i,y_j)|$$ is bounded.
\end{definition}

 \begin{definition}
  (Fr\'echet) \citep{James}.
   A function $f:[a,b] \times [c,d] \rightarrow \mathbb{R}$ is of bounded variation in the Fr\'echet sense if the sum $$ \sum_{i=0,j=0}^{m-1,n-1} \epsilon_i \overline{\epsilon}_j \triangle_{11} f(x_i,y_j) $$ is bounded for all possible nets and for all choices of $\epsilon_i = \pm1$ and $\overline{\epsilon}_j= \pm1.$
  \end{definition}

 \begin{definition}
 (Hardy-Krause) \citep{James}.
   A function $f: [a,b] \times [c,d] \rightarrow \mathbb{R}$ is said to be of bounded variation  in  the Hardy sense if it satisfies the following conditions.
   \begin{enumerate}[(i)]
   \item the same condition as that of the bounded variation in the Vitali sense

   \item for at least one fixed $\overline{x}$, the function $f(\overline{x},y)$ is of bounded variation in $y$ and for at least one $\overline{y}$, the function $ f(x, \overline{y})$ is of bounded variation in $x$.
\end{enumerate}
   \end{definition}

\begin{definition}
    (Arzel\'a) \citep{James}.
    Let $(x_i,y_j) ~~(i=0,1,2, \dots ,m; j=0,1,2,\dots,m)$ be any set of points satisfying the conditions $$ a=x_0 \le x_1 \le x_2 \le \dots \le x_m=b;$$
    $$ c =y_0 \le y_1 \le y_2 \le \dots \le y_m=d.$$
     The function $f: [a,b] \times [c,d] \rightarrow \mathbb{R}$ is said to be of bounded variation  in the Arzel\'a sense if the sum $$ \sum_{i=0}^{m-1} | \triangle f(x_i,y_i)|$$ is bounded for all such sets of points.
    \end{definition}

 \begin{definition}
 (Pierpont) \citep{James}.
   Consider a square net which covers the whole plane and has its lines parallel to the respective axes. Denote the side of each square  by $D$. No line of the net need to be coinciding with a side of the rectangle $ [a,b] \times [c,d].$ Then a finite number of the cells of the net will contain points of $[a,b] \times [c,d]$. Let us denote the oscillation of $f$ in the $r$-th  cell, regarded as a closed region by $ \omega_r$. A function $f:[a,b] \times [c,d] \rightarrow \mathbb{R}$ is said to be of bounded variation in the Pierpont sense if the sum $$ \sum_{r} D\omega_r$$ is bounded for all such nets in which $D$ is less than some fixed constant.
\end{definition}

  \begin{definition}
  (Hahn) \citep{James}.
    Consider any net in which we have $m=n$, $x_{i+1}-x_i = \frac{b-a}{m}$, and $y_{i+1}-y_i=\frac{d-c}{m}$ ($i=0,1,2,\dots , m-1$). Then there are $m^2$ congruent rectangular cells. The function $f:[a,b] \times [c,d] \rightarrow \mathbb{R}$ is said to be of bounded variation in the sense of Hahn if the sum $$ \sum_{i=1}^ {m^2} \frac{\omega_r}{m}$$ is bounded for all $m.$
\end{definition}

\begin{theorem} (\citep{James}, item (7), p. 835).
  A function  $f:[a,b] \times [c,d] \rightarrow \mathbb{R}$ is of bounded variation in the Pierpoint sense if and only if it is of bounded variation in the Hahn sense.
   \end{theorem}

\begin{definition}
   (Tonelli) \citep{James}.
   Let $f:[a,b] \times [c,d] \rightarrow \mathbb{R}$ be such that the total variation function $ \phi( \overline{x})$ is finite almost everywhere in $(a,b),$ and its Lebesgue integral over $(a,b)$ exists (finite). Further assume that  a similar condition is satisfied by $ \mu (\overline{y}).$ Then $f$ is of bounded variation in the Tonelli sense.
\end{definition}

\begin{theorem}(\citep{Raymond}, Theorem 7, p. 718). \label{BBVET5}
      A bivariate function $f:[a,b] \times [c,d] \rightarrow \mathbb{R}$ is  of bounded variation in the Arzel\'{a} sense if and only if it is expressible as the difference between two bounded functions, $f_1$ and $f_2,$ satisfying the inequalities $$ \triangle_{10} f_i(x,y) \ge 0, \quad  \triangle_{01} f_i(x,y) \ge 0 ~~ ~~ ~~~~~(i = 1,2).$$
\end{theorem}


\par
 As usual, the class of all real-valued continuous functions defined on the rectangular region $[a,b] \times [c,d]$ is denoted by  $\mathcal{C}\big([a,b] \times [c,d] \big)$ or simply by $\mathcal{C}$.  We shall denote the classes of real-valued functions defined on the region $[a,b] \times [c,d]$ satisfying  the notion of bounded variation in the sense of Vitali, Hardy, Arzel\'a, Pierpont, Fr\'echet, and Tonelli respectively by $\mathcal{V}\big(  [a,b] \times [c,d] \big)$, $\mathcal{H}\big(  [a,b] \times [c,d] \big)$, $\mathcal{A}\big(  [a,b] \times [c,d] \big)$, $\mathcal{P}\big(  [a,b] \times [c,d] \big)$, $\mathcal{F}\big(  [a,b] \times [c,d] \big)$, and $\mathcal{T}\big(  [a,b] \times [c,d] \big)$.

 \begin{theorem}(\citep{James}, item (1c), p. 846).\label{BBVIRT1}
  The following relation between various approaches to the notion of bounded variation  of a bivariate continuous function exists. $$\mathcal{H} \cap  \mathcal{C}  \subseteq    \mathcal{A} \cap \mathcal{C} \subseteq \mathcal{P}\cap  \mathcal{C}\subseteq \mathcal{T} \cap  \mathcal{C}.$$
   \end{theorem}

\subsection{Fractal dimensions}
We shall summarize two notions of fractal dimension briefly here, but refer the reader to \citep{Fal}.
\begin{definition}
  For a  non-empty subset $U$ of $\mathbb{R}^n,$ the diameter of $U$ is defined as $$|U|=\sup\big\{|x-y|: x,y \in U\big\},$$ where $|x-y|$ denotes the usual distance between $x,y$ in $\mathbb{R}^n.$ A  $\delta$-cover of $F$ is a countable  collection of sets  $\{U_i\}$ that cover $F$ such that each $U_i$ is of diameter at most $\delta.$ Suppose $F$ is a subset of $\mathbb{R}^n$ and $s$ is a non-negative real number. For any $\delta >0,$ we define $$H^s_{\delta}(F)= \inf \Big\{\sum_{i=1}^{\infty}|U_i|^s:  \{U_i\} \text{~~is ~~a} ~~\delta- \text{cover ~~of} ~F \Big\} .$$We define the $s-$dimensional Hausdorff measure of $F$ by $H^s(F)= \lim_{\delta \rightarrow 0}H^s_{\delta}(F).$
\end{definition}
 \begin{definition} 
 Let $F \subseteq \mathbb{R}^n$ and $s \ge 0.$ The Hausdorff dimension of $F$ is $$ \dim_H(F)=\inf\{s:H^s(F)=0\}=\sup\{s:H^s(F)=\infty\}.$$
 \end{definition}
 \begin{remark}
 For $s=\dim_H(F),$  $H^s(F)$ may be zero, infinite, or may satisfy $0<H^s(F)< \infty.$ A Borel set satisfying this last condition is termed an $s-$set.
 \end{remark}
 In the sequel, we shall use the following result, which reveals a fundamental property of the Hausdorff dimension.
 \begin{theorem} (\citep{Fal}, Corollary 2.4, p. 32). \label{BR1}
 Let $A \subseteq \mathbb{R}^n$ and  $f: A \to \mathbb{R}^m$.
 \begin{enumerate}[(i)]
 \item If $f:A \to \mathbb{R}^m$ is a Lipschitz map,  then $\dim_{H}\big( f(A) \big) \le \dim_{H}(A).$
 \item If $f:A \to \mathbb{R}^m$ is a bi-Lipschitz map, i.e. $$c_1|x-y| \le |f(x)-f(y) | \le c_2 |x-y|$$ for all $x,y \in A$ and $0<c_1\le c_2 < \infty$, then $\dim_{H}\big( f(A) \big) = \dim_{H}(A).$
 \end{enumerate}
 \end{theorem}

 \begin{definition} Let $F \ne \emptyset$ be a bounded subset of $\mathbb{R}^n$ and let $N_{\delta}(F)$ be the smallest number of sets of diameter at most $\delta$ which can cover $F.$ The lower box dimension and upper box dimension of $F$ respectively are defined as $$\underline{\dim}_B(F)=\varliminf_{\delta \rightarrow 0} \frac{\log N_{\delta}(F)}{- \log \delta},$$
 and
 $$\overline{\dim}_B(F)=\varlimsup_{\delta \rightarrow 0} \frac{\log N_{\delta}(F)}{- \log \delta}.$$
 If the above two are equal, we define  the box dimension of $F$ as the common value, that is,  $$\dim_B(F)=\lim_{\delta \rightarrow 0} \frac{\log N_{\delta}(F)}{- \log \delta}.$$
 \end{definition}
\subsection{Fractional integral}
Of the various formulations of fractional integral available, the Riemann-Liouville fractional integral is perhaps
the most used fractional integral, currently. In what follows, we recall this definition in the context of bivariate function; see, for instance, \citep{Samko}.

   \begin{definition}
   Let $f$ be a function defined on a closed rectangle $[a,b] \times [c,d]$ and $\alpha >0 , \beta >0.$ The (mixed) Riemann-Liouville fractional integral of $f$ is defined as $$ \mathcal{I}^{(\alpha, \beta)}f(x,y)=\frac{1}{\Gamma (\alpha) . \Gamma (\beta)} \int_a ^x \int_c ^y (x-s)^{\alpha-1} (y-t)^{\beta-1}f(s,t)~\mathrm{d}s~\mathrm{d}t .$$
   \end{definition}

\section{On Fractal dimension of the graph of a bivariate function} \label{MS1}
We begin by assembling some basic facts about the fractal dimensions of the graphs of Lipschitz functions. Some of these serve as prelude to our main results, whereas some might be of independent interest. \\
Here and in the rest of the article, we shall use the following notation. Let $A \subseteq \mathbb{R}^n$ and $f: A \to \mathbb{R}$ be a function. The graph of $f$ denoted by $G_f$ is the set
$$G_f= \Big\{\big(x,f(x)\big): x \in A  \Big\} \subseteq A \times \mathbb{R}.$$
 We shall denote by $\|.\|_2$, the Euclidean norm in the appropriate space $\mathbb{R}^m$. Some of the preparatory lemmas given below or  perhaps their special cases can be found in a  different context and in an abbreviated form elsewhere; see, for instance, \citep{LW}. However, for the sake of completeness and record, we include detailed arguments here.

 \begin{lemma} \label{BBVHDL1}
 Let $A \subseteq \mathbb{R}^n $ and $f :A \rightarrow \mathbb{R}$ be continuous on $A.$ Then $\dim_H (G_f) \ge \dim_H(A).$
 In particular, if $f :[a,b] \times [c,d] \rightarrow \mathbb{R}$ is continuous on $[a,b] \times[c,d],$ then $\dim_H (G_f) \ge 2.$
 \end{lemma}

 \begin{proof}
 Define a map $T : G_f \rightarrow A $ by $T\big((t,f(t))\big)=t, $ where $G_f \subseteq \mathbb{R}^{n+1}$ and $A\subseteq \mathbb{R}^n$ are endowed with the metric induced by the usual Euclidean norm. We have\\
 $$\big\|T\big((t,f(t))\big)-T\big((u,f(u))\big)\big\|_2 = \|t -u\|_2 \le \big\|\big(t,f(t)\big)-\big(u,f(u)\big)\big\|_2 ,$$ therefore, $T$ is a Lipschitz map. Using a basic property of the Hausdorff dimension (Cf. Theorem \ref{BR1}) we have $\dim_H \big(T(G_f)\big) \le \dim_H(G_f).$ It is easy to check that the map $T$ is surjective and hence the result.
 \end{proof}
\begin{lemma}\label{BBVEL2}
 Let $A \subseteq \mathbb{R}^n $ and $f ,g:A \rightarrow \mathbb{R}$ be continuous on $A.$ Suppose that $f$ is a Lipschitz function. Then $\dim_H (G_{f+g}) = \dim_H (G_g).$
\end{lemma}

 \begin{proof}
  We define a map $T : G_g \rightarrow G_{f+g} $ by $T\big((t,g(t))\big)=\big(t,f(t)+g(t)\big).$ It is easy to check that the map $T$ is a surjective. Though it is routine to check that $T$ is bi-Lipschitz as well, we shall include the details for sake of completeness and record. Let $M:= \max\{ \sqrt{1+2L^2},\sqrt{2}\}.$ We have
 \begin{equation*}
  \begin{aligned}
      \Big\|T\big((t,g(t))\big)-T\big((u,g(u))\big)\Big\|_2 =&~ \Big\|\big(t,f(t)+g(t)\big)-\big(u,f(u)+g(u)\big)\Big\|_2\\ =&~\sqrt{\|t-u\|_2^2+(f(t)+g(t)-f(u)-g(u))^2}\\
       \le &~\sqrt{\|t-u\|_2^2+2(f(t)-f(u))^2+2(g(t)-g(u))^2}\\
       \le &~ \sqrt{\|t-u\|_2^2+2L^2\|t-u\|_2^2+2(g(t)-g(u))^2}\\  \le &~ M\sqrt{\|t-u\|_2^2+(g(t)-g(u))^2}\\
        =&~ M \big\|(t,g(t))-(u,g(u))\big\|_2,
  \end{aligned}
 \end{equation*}
 where $L$ is a Lipschitz constant of $f$. Furthermore,
 \begin{equation*}
   \begin{aligned}
      \Big \|T\big((t,g(t))\big)-T\big((u,g(u))\big)\Big\|_2=&~ \Big\|\big(t,f(t)+g(t)\big)-\big(u,f(u)+g(u)\big)\Big\|_2.\\ =&~\sqrt{\|t-u\|_2^2+(f(t)+g(t)-f(u)-g(u))^2}.\\
        = &~ \frac{M}{M} \sqrt{\|t-u\|_2^2+(f(t)+g(t)-f(u)-g(u))^2}.\\
        \ge &~ \frac{1}{M} \sqrt{\|t-u\|_2^2 [1+2L^2]+2(f(t)+g(t)-f(u)-g(u))^2}.\\
        \ge &~ \frac{1}{M} \sqrt{\|t-u\|_2^2+2(f(t)+g(t)-f(u)-g(u))^2+2(f(t)-f(u))^2}.\\
         =&~ \frac{1}{M} \big\|(t,g(t))-(u,g(u))\big\|_2.
   \end{aligned}
  \end{equation*}
   Consequently, $T$ is a bi-Lipschitz map. Using the fact that the Hausdorff
dimension is invariant under bi-Lipschitz transformations (Cf. Theorem \ref{BR1}), we have   $\dim_H(G_{f+g}) = \dim_H (T(G_f)) = \dim_H (G_f)$.
 \end{proof}
\begin{remark}
 Since the box dimension is Lipschitz invariant, the above lemma holds for the box dimension as well.
 \end{remark}
 As is customary, we define multiplication of two functions $f,g:A \subseteq \mathbb{R}^n \to \mathbb{R}$ by $(fg)(x)=f(x)g(x).$
 \begin{lemma}
  Let $A \subseteq \mathbb{R}^n $ and $f ,g:A \rightarrow \mathbb{R}$ be continuous on $A.$ Suppose that $f$ is a Lipschitz function. Then $  \dim_H (G_{fg}) \le \dim_H (G_g).$

  \end{lemma}
  \begin{proof}
  The mapping $T : G_g \rightarrow G_{fg} $ defined by $$T\big((t,g(t))\big)=\big(t,f(t)g(t)\big)$$ is surjective and Lipschitz with a Lipschitz constant $M= \max\{ \sqrt{1+2M_gL^2},\sqrt{2}M_f\},$ where $L$ is a Lipschitz constant of $f,$ $M_f=\|f\|_{\infty},$ and $M_g=\|g\|_{\infty}.$
  \end{proof}
  \begin{remark}
  Since the box dimension is Lipschitz invariant, the above lemma is also true for the box dimension.
  \end{remark}

  \begin{remark}
    In the previous lemma, we may not get equality in general. To see this, let us take $g$ to be the Weierstrass function with the Hausdorff dimension strictly greater than one (See \citep{Shen}) and $f$ to be the zero function. Then, we obtain $1=\dim_H(G_{fg}) < \dim_H (G_g).$
  \end{remark}

 \begin{lemma}(\citep{Fal}, Corollary 7.4, p. 102).
 Let $D,E \subseteq \mathbb{R}^n.$ We have
 \begin{enumerate}[(i)]
 \item $\overline{\dim}_B (D \times E) \le \overline{\dim}_B (D)+\overline{\dim}_B (E).$
 \item If $\dim_H (D)=\overline{\dim}_B (D)$, then $\dim_H(D \times E) =\dim_H (D) +\dim_H (E).$
 \end{enumerate}
 \end{lemma}

\begin{lemma} \label{BBVEL4}
Let $f :[c,d] \rightarrow \mathbb{R}$ be continuous and let $ a< b.$ Define a set $E= \big\{(x,y,f(y)): x \in [a,b], y \in [c,d]  \big\}.$ Then, $\dim_H (E)= \dim_H (G_f) + 1$ and $ \overline{\dim}_B(E) \le   \overline{\dim}_B(G_f) +1.$
\end{lemma}
\begin{proof}
First let us note that the set $E$ is equal to $[a,b] \times G_f.$
Since $\dim_H ([a,b])=\overline{\dim}_B ([a,b]),$ by the previous lemma it follows that $\dim_H (E)= \dim_H (G_f) + 1$ and $ \overline{\dim}_B(E) \le   \overline{\dim}_B(G_f) +1.$
\end{proof}
Next we shall study the Hausdorff dimension of the graphs of some special type of bivariate functions. Let $f :[a,b] \rightarrow \mathbb{R}$ and $g:[c,d] \rightarrow \mathbb{R}$ be continuous maps. Define $h_1,h_2 :[a,b] \times [c,d] \rightarrow \mathbb{R}$ by $$h_1(x,y)=f(x)+g(y), \quad \text{and} \quad h_2(x,y)=f(x)g(y).$$
 \begin{lemma}  \label{BBVEL3}
   Let $f :[a,b] \rightarrow \mathbb{R}$ be a Lipschitz map and $g:[c,d] \rightarrow \mathbb{R}$ be continuous. Then, $  \dim_H (G_{h_1}) = \dim_H (G_g)+1$ and   $\dim_H (G_{h_2})\le \dim_H(G_g)+1.$
   \end{lemma}
  \begin{proof}
   Define a set $E= \big\{(x,y,g(y)): x \in [a,b], y \in [c,d]  \big\}.$ The mapping $T : E \rightarrow G_{h_1} $  defined by $$T\big((t,u,g(u))\big)=(t,u,f(t)+g(u)).$$
   is a surjective bi-Lipschitz map, whence the previous lemma implies that   $\dim_H(G_{h_1}) = \dim_H \big(T(E)\big) = \dim_H (G_g)+1.$ A similar proof for the other conclusion.
   \end{proof}
   \begin{remark} \label{BBVERM}
   On similar lines using lemma \ref{BBVEL4}, we have $  \overline{\dim}_B (G_{h_1}) \le \overline{\dim}_B (G_g)+1$ and   $\overline{\dim}_B (G_{h_2})\le \overline{\dim}_B (G_g)+1.$
   \end{remark}
\begin{definition}
Let $A \subset \mathbb{R}^2$ be a closed bounded  rectangle and $f: A \to \mathbb{R}$. The maximum range of  $f$ over the rectangle $A$ is defined as
$$ R_f[A]= \sup_{(t,u),(x,y) \in A} |f(t,u)-f(x,y)|.$$
\end{definition}
As indicated in the introductory section, next we shall provide a bivariate analogue of Proposition $11.1$  in Falconer \citep{Fal}.
\begin{lemma} \label{BBVL2}
Let $f :[a,b] \times [c,d] \rightarrow \mathbb{R}$ be continuous. Suppose that $ 0 < \delta < \min\{b-a,d-c\},$ $ \frac{b-a}{\delta} \le m \le 1+ \frac{b-a}{\delta }$ and $ \frac{d-c}{\delta} \le n \le 1+ \frac{d-c}{\delta }$ for some $m ,n \in \mathbb{N}.$ If $N_{\delta}(G_f)$ is the number of $\delta-$cubes that intersect the graph of $f,$  then $$ \frac{1}{\delta} \sum_{j=1}^{n} \sum_{i=1}^{m} R_f[A_{ij}] \le N_{\delta}(G_f) \le 2mn + \frac{1}{\delta} \sum_{j=1}^{n} \sum_{i=1}^{m} R_f[A_{ij}].$$
\end{lemma}
\begin{proof}
The number of cubes of side length $\delta$ in the part above $A_{ij}$ that intersect the graph of $f$ is at least $\frac{R_f[A_{ij}]}{ \delta} $ and at most  $2+ \frac{R_f[A_{ij}]}{ \delta},$ using that $f$ is continuous. Summing over all such parts gives the desired bounds.
\end{proof}
The preceding result may be applied to functions satisfying a
H\"{o}lder condition to obtain the following corollary.
\begin{corollary}
Let $f:[a,b] \times [c,d] \rightarrow \mathbb{R}$ be a continuous function.
\begin{enumerate}[(i)]
\item  Suppose
\begin{equation} \label{HCeq1}
|f(t,u)-f(x,y)| \le c ~~||(t,u)-(x,y)||_2^s, ~~~~ \forall ~~(t,u),(x, y) \in [a,b] \times [c,d],
 \end{equation}
 where $c>0$ and $0 \le s \le 1.$ Then $\dim_H (G_f) \le \overline{\dim}_B (G_f) \le 3-s.$ The conclusion remains true if the H\"older condition in (\ref{HCeq1}) holds when $||(t,u)-(x,y)||_2 < \delta $ for some $\delta >0.$
 \item Suppose that there are numbers $c>0, \delta_0 >0$ and $0 \le s \le 1$ with the following property: for each $(x,y) \in [a,b] \times [c,d] $ and $ 0< \delta <\delta_0$ there exists $(t,u)$ such that $||(t,u)-(x,y)||_2 \le \delta$ and

     \begin{equation} \label{HCeq2}
     |f(t,u)-f(x,y)| \ge c \delta^s.
     \end{equation}
     Then $ \underline{\dim}_B (G_f) \ge 3 - s.$
\end{enumerate}
\end{corollary}
\begin{proof}
\begin{enumerate}[(i)]
\item Since $f$ satisfies the H\"older condition in (\ref{HCeq1}), we have $$R_f[A_{ij}] \le  c(\sqrt{2}\delta)^s .$$ Therefore from the previous lemma, we obtain $$N_{\delta}(G_f) \le 2mn + cmn {2}^\frac{s}{2} \delta^{s-1} .$$
The upper box dimension of $G_f$ can be estimated as $$ \overline{\lim}_{\delta \rightarrow 0} \frac{\log N_{\delta}(G_f)}{-\log \delta} \le \lim_{\delta \rightarrow 0}\dfrac{\log(2mn + cmn 2^{\frac{s}{2}} \delta^{s-1})}{- \log \delta},$$
which provides $$\overline{\lim}_{\delta \rightarrow 0} \frac{\log N_{\delta}(G_f)}{-\log \delta} \le 3-s .$$

\item On similar lines, (\ref{HCeq2}) implies that $R_f[A_{ij}] \ge c (\sqrt{2}\delta) ^s.$ The previous lemma now yields $ N_{\delta}(G_f) \ge  cmn 2^{\frac{s}{2}} \delta^{s-1}.$ On similar lines, we estimate the lower box dimension of $G_f$  to arrive at $\underline{\dim}_B(G_f) \ge 3- s$, completing the proof.
\end{enumerate}
\end{proof}
\begin{theorem}\label{BBVMT1}
If $f :[a,b] \times [c,d] \rightarrow \mathbb{R}$ is continuous and of bounded variation in the sense of Hahn (or Pierpont). Then $\dim_B(G_f) =2.$
\end{theorem}
\begin{proof}
We observe that $\underline{\dim}_B G_f  \ge \dim_H G_f \ge 2$ by Lemma \ref{BBVHDL1}.
Let $ 0 < \delta < \min\{b-a,d-c \}$ and $ \frac{b-a}{\delta} \le m \le 1+ \frac{b-a}{\delta }$ and $ \frac{d-c}{\delta} \le n \le 1+ \frac{d-c}{\delta }$ for some $m ,n \in \mathbb{N}.$ From Lemma \ref{BBVL2} we know that the number of $\delta-$cubes that intersect the graph of $f$ is
$$N_{\delta}(G_f) \le 2mn + \frac{1}{\delta} \sum_{j=1}^{n} \sum_{i=1}^{m} R_f[A_{ij}].$$
Since $f$ is of bounded variation in the sense of Pierpont, by definition, we have $ \sum_{j=1}^{n} \sum_{i=1}^{m}  \delta R_f[A_{ij}]$ is bounded for all $\delta$ where $ \delta_0 > \delta >0$ for some fixed $\delta_0 >0.$
 To calculate the box dimension of $G_f$, one deals with sufficiently  small $\delta -$cover of $G_f$ and hence we may assume that $ \sum_{j=1}^{n} \sum_{i=1}^{m} \delta R_f[A_{ij}]$ is bounded for all sufficiently small $\delta >0.$ That is, there exists $K >0$ such that $$\sum_{j=1}^{n} \sum_{i=1}^{m} R_f[A_{ij}] =  \sum_{j=1}^{n} \sum_{i=1}^{m} \omega_{ij} \le K.\frac{1}{\delta}$$ for sufficiently small $\delta>0$.
 Consequently,
$$ \overline{\lim}_{\delta \rightarrow 0} \frac{\log N_{\delta}(G_f)}{-\log \delta} \le \lim_{\delta \rightarrow 0}\frac{\log(2mn + \frac{1}{\delta}K\frac{1}{\delta})}{-\log \delta},$$
which on calculation produces $$ \overline{\dim}_B(G_f)= \overline{\lim}_{\delta \rightarrow 0} \frac{\log N_{\delta}(G_f)}{-\log \delta} \le 2 .$$
\end{proof}
\begin{remark}
Using the interconnection between the various notions of bounded variation of a continuous function (Cf. Theorem \ref{BBVIRT1}) and the previous theorem one can conclude that if $f :[a,b] \times [c,d] \rightarrow \mathbb{R}$ is continuous and  of bounded variation in the sense of Arzela or Hardy,  then $\dim_H(G_f)=\dim_B (G_f)=2.$
\end{remark}
In what follows, we shall provide some simple examples for a function $f: [a,b] \times [c,d] \to \mathbb{R}$ such that $\dim_H (G_f) = \dim_B(G_f)=2$, but $f$  is not of bounded variation.
\begin{example}
Define a function $f$ on $[0,1]^2$ by $f(x,y)=1,$ for $x,y \in \mathbb{Q}$ and $f(x,y)=0,$ otherwise. This function $f$ is of bounded variation in the sense of Tonelli but not in the sense of Pierpont \citep{James}. Note that $$2 \le \dim_H(G_f) \le    \underline{\dim}_B(G_f) \le  \overline{\dim}_B(G_f).$$ Next to bound
$\overline{\dim}_B(G_f)$, let us observe that
$$G_f= \big\{(x,y,1):(x,y) \in [0,1]^2\cap (\mathbb{Q} \times \mathbb{Q})\big\} \bigcup \big\{(x,y,0)(x,y) \in [0,1]^2\cap (\mathbb{Q}\times \mathbb{Q})^c\big\}:=A_1 \cup A_2$$
Since $\overline{\dim}_B(G_f) = \max_{i=1,2} \overline{\dim}_B(A_i) = \max_{i=1,2} \overline{\dim}_B(\overline{A_i})$, it follows that
$\dim_B(G_f)=\dim_H(G_f)= 2.$ The reader can compare this with Theorem \ref{BBVMT1}.
\end{example}
\begin{example}
Define a function $f$ on $[0,1]^2$ as follows
 \begin{equation*} f(x,y) =
       \begin{cases}
     1~~ \text{if}~~ x+y=1 \\
          0 ~~ \text{otherwise}.
      \end{cases}
   \end{equation*}
    This function  is discontinuous and it is of bounded variation in sense of Arzel\'a but not in sense of Fr\'echet \citep{James}. We write the graph of the function $f$ as
\begin{equation*}
   \begin{aligned}
    G_f = & \big\{(x,y,1): x,y \in [0,1] ~~\text{and} ~~x+y =1\big\}  \cup \big\{(x,y,0): (x,y) \in [0,1] ~~\text{and} ~~x+y < 1\big\}\\
    & \cup  \big\{(x,y,0): x,y \in [0,1] ~~\text{and} ~~x+y > 1\big\}\\
    &:=B_1 \cup B_2 \cup B_3.
             \end{aligned}
               \end{equation*}
  Let us recall that $2 \le \dim_H(G_f) \le \overline{\dim}_B(G_f) \le \overline{\dim}_B(G_f)$ and note that $\overline{\dim}_B(B_1)=1,$ $\overline{\dim}_B(B_2)=2$ and $\overline{\dim}_B(B_3)=2.$ Since $\overline{\dim}_B$ is finitely stable (see \citep{Fal}), we get $\overline{\dim}_B(G_f)=2.$ Therefore, $\dim_B(G_f)=\dim_H(G_f)=2.$
\end{example}
\begin{example}
Define a function $f$ on $[0,1]^2$ as follows $f(x,y)=0$ whenever $x < y ,$ otherwise $f(x,y)=1.$ This function is of bounded variation in sense of Tonelli but not in sense of Hardy \citep{James}. We write the graph of the function $f$ as $$ G_f=\{(x,y,0): x,y \in [0,1]  ~~ \text{and} ~~ x<y \} \cup \{(x,y,1): x,y \in [0,1] ~~~\text{and} ~~x \ge y \}:=C_1 \cup C_2$$ Note that $\dim_H(G_f)=2.$ We write $2= \dim_H(G_f) \le \underline{\dim}_B(G_f) \le \overline{\dim}_B(G_f).$
Furthermore, we have  $\overline{\dim}_B(C_1) =2$ and $\overline{\dim}_B(C_2)=2.$
Since $\overline{\dim}_B$ is finitely stable, we deduce that $\overline{\dim}_B(G_f)=2$ and hence that $\dim_B(G_f)=\dim_H(G_f)=2.$
\end{example}

\begin{example}
Define a function $f$ on $[0,1]^2$ as follows $f(x,y)=x \sin(\frac{1}{x})$ whenever $x \ne 0$, otherwise $f(x,y)=0.$ The function defined above is of bounded variation in the sense of Vitali but not in sense of Pierpont \citep{James}. Let us mention that if $g:[0,1] \to \mathbb{R}$ is continuous and
 has one unbounded variation point on $[0,1]$, then $\dim_B(G_g)=\dim_H(G_g)=1$: see, for instance,  \citep{Liang2}. Now using Remark \ref{BBVERM}, and that    the function
   \begin{equation*} g(x) =
       \begin{cases}
     x \sin \frac{1}{x}~~ \text{if}~~ x \neq 0\\
          0 ~~ \text{if}~~ x=0
      \end{cases}
   \end{equation*}
 has only one point of unbounded variation on $[0,1]$, we deduce  that  $\dim_B(G_f)=\dim_H(G_f)=2.$
\end{example}

\section{Dimension of graph of fractional integral of continuous function} \label{MS2}
In this section we consider $ 0 \le a <b <\infty$ and $ 0 \le c < d <\infty.$
\begin{theorem}

If $f$ is a bounded function on $[a,b] \times [c,d] $ and $\alpha >0, \beta >0$, then the Riemann-Liouville fractional integral $I^{(\alpha,\beta)}f$ is bounded.
\end{theorem}
\begin{proof}
Since $f$ is bounded, there exists $M>0$ such that $|f(x,y)|\le M,~~\forall~(x,y) \in [a,b] \times [c,d].$ For each fixed $(x,y) \in [a,b] \times [c,d],$ we have
\begin{equation*}
    \begin{aligned}
     |I^{(\alpha,\beta)}f(x,y)|=&~ \Big|\frac{1}{\Gamma (\alpha) \Gamma(\beta) } \int_a ^x \int_c ^y (x-s)^{\alpha-1} (y-t)^{\beta-1}f(s,t)~\mathrm{d}s~\mathrm{d}t\Big|\\
     \le &~ \frac{1}{\Gamma (\alpha) \Gamma(\beta)} \int_a ^x  \int_c ^y \big|(x-s)^{\alpha-1} (y-t)^{\beta-1}\big|~~|f(s,t)|\mathrm{d}s ~\mathrm{d}t\\
     \le&~ \frac{M}{\Gamma (\alpha) \Gamma(\beta)} \int_a ^x  \int_c ^y |(x-s)^{\alpha-1} (y-t)^{\beta-1}|\mathrm{d}s~ \mathrm{d}t\\
    \le &~  \frac{M}{ \Gamma (\alpha)\Gamma(\beta)} \frac{(x-a)^{\alpha} (y-c)^{\beta}}{ \alpha \beta}.
     \end{aligned}
               \end{equation*}
Consequently,
\begin{equation*}
    \begin{aligned}
     |I^{(\alpha,\beta)}f(x,y)| \le &~ \frac{M}{ \Gamma (\alpha+1)\Gamma (\beta+1)} (b-a)^{\alpha} (d-c)^{\beta}, ~~\forall~ (x,y) \in [a,b] \times [c,d],
     \end{aligned}
               \end{equation*}
  completing the proof.
\end{proof}
\begin{theorem}\label{BBVT9}
If $f$ is a continuous function on $[a,b] \times [c,d] $ and $\alpha >0$ and $\beta >0,$
 then $ I^{(\alpha,\beta)}f$ is continuous on $[a,b] \times [c,d].$
\end{theorem}
\begin{proof}
The proof follows by standard lines as given below. Let $a <x \le x+h < b$ and $c <y \le y+k <d.$ Then,
\begin{equation}\label{BBVFIeq1}
    \begin{aligned}
     I^{(\alpha,\beta)}f(x+h,y+k) - I^{(\alpha,\beta)}f(x,y)=&~ \frac{1}{\Gamma \alpha .\Gamma \beta } \int_a ^{x+h} \int_c ^{y+k} (x+h-s)^{\alpha-1} (y+k-t)^{\beta-1}f(s,t)~\mathrm{d}s~\mathrm{d}t \\
     &- \frac{1}{\Gamma \alpha .\Gamma \beta } \int_a ^x \int_c ^y (x-s)^{\alpha-1} (y-t)^{\beta-1}f(s,t)~\mathrm{d}s~\mathrm{d}t.\\
     = I_1+ I_2 +I_3+I_4 -I_5,
     \end{aligned}
               \end{equation}
     where
  \begin{equation*}
    \begin{aligned}
     I_1 =&~ \frac{1}{\Gamma(\alpha) \Gamma(\beta)} \int_a ^{a+h} \int_c ^{c+k} (x+h-s)^{\alpha-1} (y+k-t)^{\beta-1}f(s,t)~\mathrm{d}s~\mathrm{d}t, \\
    I_2= &~ \frac{1}{\Gamma(\alpha)\Gamma (\beta) } \int_{a+h} ^{x+h} \int_c ^{c+k} (x+h-s)^{\alpha-1} (y+k-t)^{\beta-1}f(s,t)~\mathrm{d}s~\mathrm{d}t, \\
     I_3= &~ \frac{1}{\Gamma(\alpha)\Gamma(\beta) } \int_a ^{a+h} \int_{c+k} ^{y+k} (x+h-s)^{\alpha-1} (y+k-t)^{\beta-1}f(s,t)~\mathrm{d}s~\mathrm{d}t, \\
     I_4=&~ \frac{1}{\Gamma (\alpha)(\Gamma \beta) } \int_{a+h} ^{x+h} \int_{c+k} ^{y+k} (x+h-s)^{\alpha-1} (y+k-t)^{\beta-1}f(s,t)~\mathrm{d}s~\mathrm{d}t,\\
    I_5=&~ \frac{1}{\Gamma(\alpha)\Gamma(\beta)} \int_a ^x \int_c ^y (x-s)^{\alpha-1} (y-t)^{\beta-1}f(s,t)~\mathrm{d}s~\mathrm{d}t.\\
          \end{aligned}
               \end{equation*}
Using the change of variable $u= s-h$ and $v=t-k$ in the integral $I_4$, Eq. (\ref{BBVFIeq1}) yields
\begin{equation}
    \begin{aligned} \label{BBVEQ5}
     I^{(\alpha,\beta)}f(x+h,y+k) - I^{(\alpha,\beta)}f(x,y)
     =I_1+I_2+I_3+I_6, \end{aligned}
          \end{equation}
    where
        \begin{equation*}
            \begin{aligned}
          I_6= \frac{1}{\Gamma \alpha .\Gamma \beta } \int_a ^x \int_c ^y (x-s)^{\alpha-1} (y-t)^{\beta-1}[f(s+h,t+k)-f(s,t)]~\mathrm{d}s~ \mathrm{d}t.\\
          \end{aligned}
               \end{equation*}
Since $f$ is continuous  on $[a,b] \times [c,d]$, there exits $M>0$ such that $|f(s,t)| \le M$ for every $(s,t) \in [a,b] \times [c,d] $.
 Again using the fact that a real-valued continuous function on a closed bounded interval in $\mathbb{R}$ is bounded, for a suitable constant $M_1$, we obtain
 $$|I_1| \le M_1 hk.$$
 With $M_2= \max_{c\le t \le c+k} (y+k-t)^{\beta-1} \frac{M (b-a)^\alpha}{\Gamma (\alpha+1) \Gamma (\beta)} $, performing the required integration we have
 $$|I_2| \le M_2 k.$$
 Similarly, by defining suitable $M_3$ one can bound $I_3$ as
 $$|I_3| \le M_3 h.$$
 Since $f$ is uniformly continuous, for a given $\epsilon >0,$ there exits $ \delta_1 > 0 $ such that  for $||(s,t)-(z,w)||_2< \delta_1$
$$|f(s,t)-f(z,w)| < \frac{\epsilon \Gamma(\alpha+1) \Gamma (\beta+1)}{4(b-a)^\alpha (d-c)^\beta}.$$
 Consequently,  we gather that
 \begin{equation*}
        \begin{aligned}
         |I^{(\alpha,\beta)}f(x+h,y+k) - I^{(\alpha,\beta)}f(x,y)| \le  M_1 hk +M_2 k+M_3 h
                                 + \frac{\epsilon}{ 4},
         \end{aligned}
                   \end{equation*}
 from which the continuity of $I^{(\alpha,\beta)}f$  follows.
\end{proof}
In the upcoming lemma, $g$ is monotone stands for $\triangle_{10} g \ge 0$ and $ \triangle_{01} g \ge 0.$
\begin{lemma} \label{BBVL12}
Let $f:[a,b] \times [c,d] \rightarrow \mathbb{R} $ be of bounded variation in the sense of Arzel\'a. Then the following hold.
\begin{enumerate}[(i)]
\item If $f(a,c) \ge 0,$ then there exist monotone functions $f_1$ and $f_2$ such that $f=f_1 -f_2$ with $f_1(a,c) \ge 0$ and $f_2(a,c)=0.$
 \item If $f(a,c) < 0,$ then there exist monotone functions $f_1$ and $f_2$ such that $f=f_1 - f_2$ with $f_1(a,c) = 0$ and $f_2(a,c) >0.$
\end{enumerate}
\end{lemma}
\begin{proof}
Using Theorem \ref{BBVET5} we can write the function $f$ in the  form, $f=g_1-g_2,$ where $g_1$ and $g_2$ are monotone functions. Define functions $f_1$ and $f_2$ as $f_1(x,y)=g_1(x,y)+f(a,c)-g_1(a,c)$ and $f_2(x,y)=g_2(x,y)+f(a,c)-g_1(a,c). $ It is obvious that $f_1$ and $f_2$ both are still monotone and satisfy the required conditions. For the claim in item (2), we take $f_1(x,y)=g_1(x,y)-f(a,c)-g_2(a,c)$ and $f_2(x,y)=g_2(x,y)-f(a,c)-g_2(a,c)$.
\end{proof}
\begin{theorem}\label{BBVT10}

If $f$ is of bounded variation on $[a,b] \times [c,d]$ in the sense of Arzel\'a and $\alpha >0,\beta >0,$ then $I^{(\alpha,\beta)}f$ is of bounded variation on $[a,b] \times [c,d]$ in the same sense.
\end{theorem}
\begin{proof}
Since $f$ is of bounded variation  in the Arzel\'a's sense, $f$ can be written as a difference of two monotone increasing functions. That is, $f(x,y)=f_1(x,y)-f_2(x,y), ~~~~\forall~~ (x,y) \in [a,b] \times [c,d],$ where $f_1$ and $f_2$ are monotone functions. We shall  show that $I^{(\alpha,\beta)}f$ is a difference of two monotone functions.
\begin{enumerate}[(i)]
  \item If $f(a,c) \ge 0,$ by the preceding lemma, we can choose $f_1(a,c) \ge 0$ and $f_2(a,c) = 0.$ Define functions $F_1$ and $F_2$ as follows:  $$F_1(x,y):= I^{(\alpha,\beta)}f_1(x,y), \quad F_2(x,y):= I^{(\alpha,\beta)}f_2(x,y).$$ Linearity of the fractional integral yields, $I^{(\alpha,\beta)}f(x,y)=F_1(x,y)-F_2(x,y).$ Hence it remains only to show that $F_1$ and $F_2$ are monotone functions. For this, let $a\le x_1 \le x_2 \le b$ and $ c \le y \le d.$
\begin{equation*}
    \begin{aligned}
     F_1(x_2,y)-F_1(x_1,y)=&~ \frac{1}{\Gamma (\alpha) \Gamma(\beta) } \int_a ^{x_2} \int_c ^{y} (x_2-s)^{\alpha-1} (y-t)^{\beta-1}f_1(s,t)~\mathrm{d}s~\mathrm{d}t \\
          &- \frac{1}{\Gamma(\alpha) \Gamma (\beta) } \int_a ^{x_1} \int_c ^y (x_1-s)^{\alpha-1} (y-t)^{\beta-1}f_1(s,t)~\mathrm{d}s~\mathrm{d}t.\\
          =&~ \frac{1}{\Gamma(\alpha)\Gamma(\beta) } \int_a ^{a+x_2-x_1} \int_c ^{y} (x_2-s)^{\alpha-1} (y-t)^{\beta-1}f_1(s,t)~\mathrm{d}s~\mathrm{d}t \\
          &+ \frac{1}{\Gamma(\alpha)\Gamma(\beta) } \int_{a+x_2-x_1} ^{x_2} \int_{c} ^{y} (x_2-s)^{\alpha-1} (y-t)^{\beta-1}f_1(s,t)~\mathrm{d}s~\mathrm{d}t \\
               &- \frac{1}{\Gamma(\alpha)\Gamma(\beta)} \int_a ^{x_1} \int_c ^y (x_1-s)^{\alpha-1} (y-t)^{\beta-1}f_1(s,t)~\mathrm{d}s~\mathrm{d}t.
      \end{aligned}
               \end{equation*}
Applying the change of variable $s =x_2-x_1 + u $ in the second integral above, we get
\begin{small}
\begin{equation*}
    \begin{aligned}
    F_1(x_2,y)-F_1(x_1,y)
              =& \frac{1}{\Gamma (\alpha) \Gamma(\beta) } \int_a ^{a+x_2-x_1} \int_c ^{y} (x_2-s)^{\alpha-1} (y-t)^{\beta-1}f_1(s,t)\mathrm{d}s~\mathrm{d}t \\
             &+ \frac{1}{\Gamma (\alpha)\Gamma(\beta) } \int_a ^{x_1} \int_c ^y (x_1-s)^{\alpha-1} (y-t)^{\beta-1}[f_1(s+x_2-x_1,t)-f_1(s,t)]\mathrm{d}s~\mathrm{d}t.
                   \end{aligned}
               \end{equation*}
               \end{small}
Since $ s+x_2-x_1 \ge s, $ $ f_1(a,c)\ge 0$ and $f_1$ is monotone,  all terms under the integration are non-negative. Hence, $F_1(x_2,y)-F_1(x_1,y) \ge 0,$ that is, $\triangle_{10}F_1 \ge 0.$ On similar lines, for  $ c \le y_1 \le y_2 \le d$ and $a \le x \le b,$ we have $\triangle_{01}F_1 \ge 0.$ Therefore, $F_1$ is monotone. In a similar way, one can show that $F_2$ is also a monotone function.
\item
If $f(a,c)<0,$ by Lemma \ref{BBVL12}, one can choose monotone functions $f_1$ and $f_2$ satisfying $f_1(a,c)=0$ and $f_2(a,c)>0.$
Following the proof in case (i), we have, $F_1$ and $F_2$ are monotone functions.
\end{enumerate}
\end{proof}
Next theorem shows that the box dimension and Hausdorff dimension of the Riemann-Liouville  fractional integral of a continuous function of bounded variation in the sense of Arzel\'a is $2$. The proof follows at once from Theorems \ref{BBVMT1}, \ref{BBVT9} and \ref{BBVT10}.
\begin{theorem}
If $f:[a,b] \times [c,d] \to \mathbb{R}$ is a continuous function of bounded variation in the Arzel\'a sense, then $\dim_B(G_{I^{(\alpha,\beta)}f})=\dim_H(G_{I^{(\alpha,\beta)}f})=2.$
\end{theorem}

\begin{remark}
Suppose that $g:[a,b] \rightarrow \mathbb{R}$ is continuous. We define a bivariate function $f :[a,b] \times [c,d] \rightarrow \mathbb{R}$ by $f(x,y)=g(x),$ which is continuous. We know that $$ I^{(\alpha,\beta)}f(x,y)= \frac{1}{\Gamma \alpha .\Gamma \beta } \int_a ^x \int_c ^y (x-s)^{\alpha-1} (y-t)^{\beta-1}f(s,t)\mathrm{d}s\mathrm{d}t .$$ For $\beta=1,$ we obtain $$ I^{(\alpha,\beta)}f(x,y)= \frac{1}{\Gamma \alpha  } \int_a ^x \int_c ^y (x-s)^{\alpha-1} f(s,t)\mathrm{d}s\mathrm{d}t .$$ By the definition of $f$, we get $$ I^{(\alpha,\beta)}f(x,y)=  \frac{y-c}{\Gamma \alpha  } \int_a ^x  (x-s)^{\alpha-1} g(s)\mathrm{d}s.$$ Finally, we have a relation between the Riemann-Liouville fractional integral of $g$ and mixed Riemann-Liouville fractional integral of $f$ as
 $$ I^{(\alpha,\beta)}f(x,y)= (y-c)~~I^{\alpha}g(x).$$ By Remark \ref{BBVERM}, $\overline{\dim}_B (G_{I^{(\alpha,\beta)}f}) \le \overline{\dim}_B(G_{I^{\alpha}g})+1$. From Theorem \ref{BBVET5}, it follows  that if $g$ is of bounded variation on $[a,b]$, then $f$ is so in the sense of Arzel\'a. From Liang \citep{Liang} it follows that $\dim_B(G_{I^{\alpha}g})=\dim_H(G_{I^{\alpha}g})=1$, and hence $\dim_B(G_{I^{(\alpha,\beta)}f})= \dim_H(G_{I^{(\alpha,\beta)}f}) =2.$ Our previous theorem provides the value of the Hausdorff and box dimension of the graph of the Riemann-Liouville fractional integral of a more general continuous function of bounded variation (in the sense of Arzel\'a).

\end{remark}

Let us conclude with a few remarks. The main theorems in this paper present  bivariate analogues of theorems in Liang \citep{Liang} with suitable interpretations for the notion of bounded variation. However, we should admit that it remains open  whether or not the results hold for  a bivariate continuous function of bounded variation according to the definitions other than those mentioned in our main results and remarks, for instance, if the function is of bounded variation  in the sense of Tonelli. The multivariate analogues can be considered for a future work.


\subsection*{Acknowledgments}
The first author thanks the University Grants
Commission (UGC), India for financial support in the form of a Junior Research Fellowship.

\end{document}